\documentclass[oneside,12pt]{amsart}
\usepackage{eurosym}
\usepackage{amsfonts}
\usepackage{amsmath,amssymb,amsthm,color}
\usepackage{amsopn}
\usepackage{latexsym}
\usepackage{float}
\usepackage{bm}
\usepackage[utf8]{inputenc}
\usepackage[notref,notcite]{showkeys}
\usepackage{xcolor}

\newcommand{\mg}{\mathfrak g }

\newcommand{\mn}{\mathfrak n }

\newcommand{\mz}{\mathfrak z }

\newcommand{\mv}{\mathfrak v }
\newcommand{\mh}{\mathfrak h }

\newcommand{\so}{\mathfrak{so} }

\newcommand{\bil}{g}
\newcommand{\lela}{ g(}
\newcommand{\rira}{)}
\newcommand{\lra}{\longrightarrow}

\renewcommand{\Im}{\rm Im\ }
\newcommand{\R}{\mathbb R}

\DeclareMathOperator{\End}{End}
\DeclareMathOperator{\Der}{Der}

\DeclareMathOperator{\ad}{ad}

\numberwithin{equation}{section}

 \newtheorem{teo}{Theorem}[section]
 \newtheorem{pro}[teo]{Proposition}
 \newtheorem{cor}[teo]{Corollary}
 \newtheorem{lm}[teo]{Lemma}
 \newtheorem{defi}[teo]{Definition}

 \theoremstyle{definition}
 \newtheorem{ex}[teo]{Example}
 \newtheorem{remark}[teo]{Remark}

\newcommand{\nc}{\newcommand}
\nc{\Iso}{\operatorname{Iso}}
 \nc{\iso}{\mathfrak{iso}}
 \nc{\sso}{\mathfrak{so}}
\nc{\Ad}{\operatorname{Ad}} 
\nc{\Sym}{\mathrm{Sym}}
 \nc{\pr}{\operatorname{pr}} 
  \nc{\dd}{\mathrm{d}} 
 \nc{\Dera}{\operatorname{Dera}} \nc{\Auto}{\operatorname{Auto}}
 
\begin{document}

\title{Higher degree Killing forms on $2-$step nilmanifolds}
\author{Viviana del Barco}
\address{Université Paris-Saclay, CNRS, Laboratoire de mathématiques d'Orsay, 91405, Orsay, France and Universidad Nacional de Rosario, CONICET, 2000, Rosario, Argentina}
\email{viviana.del-barco@math.u-psud.fr}

\author{Andrei Moroianu}
\address{Université Paris-Saclay, CNRS,  Laboratoire de mathématiques d'Orsay, 91405, Orsay, France}
\email{andrei.moroianu@math.cnrs.fr}

\begin{abstract} We study left-invariant Killing forms of arbitrary degree on simply connected $2-$step nilpotent Lie groups endowed with left-invariant Riemannian metrics, and classify them when the center of the group is at most two-dimensional. 
\end{abstract}

\subjclass[2010]{53D25, 22E25, 53C30} 
\keywords{Killing forms, $2-$step nilpotent Lie groups.}
\maketitle

\section{Introduction}

Killing forms on Riemannian manifolds are differential forms whose covariant deri\-vative with respect to the Levi-Civita connection is totally skew-symmetric \cite{Se03}. They generalize to higher degrees the concept of Killing vector fields (or infinitesimal isometries). Examples of Riemannian manifolds with non-parallel Killing $k-$forms are quite rare for $k\ge2$ \cite{lau,bms,au,Se03,s,y75}. 

Motivated by this fact, we have started in \cite{dBM2} a systematic study of left-invariant Killing $k-$forms on simply connected $2-$step nilpotent Lie groups  endowed with left-invariant Riemannian metrics, which will be referred to as nilmanifolds in this paper. We have shown that when $k=2$ or $k=3$, we can restrict ourselves to the de Rham irreducible case, for which Killing $2-$forms exist if and only if the Lie algebra of the group has a bi-invariant orthogonal complex structure, and Killing $3-$forms exist if and only if the Riemannian Lie group is naturally reductive, according to the  characterization given by C. Gordon \cite{GO2}.

In the present paper we continue this study and generalize the results from \cite{dBM2} in several directions, as well as those regarding the Riemannian structure of nilpotent Lie groups given in \cite{dBM}. In Theorem \ref{teo:dR} we show that the de Rham decomposition of nilmanifolds is closely related to the decomposition of the underlying metric Lie algebra in an orthogonal direct sum of ideals. This generalizes to any nilpotency degree the results obtained in \cite[Appendix A]{dBM} for $2-$step nilpotent Lie groups.

Next, in Proposition \ref{p32} we show that every left-invariant Killing form on a product of Riemannian Lie groups is a sum of Killing forms on the factors and a parallel form. This extends to forms of arbitrary degree and to arbitrary Lie groups our previous results \cite[Propositions 4.4 and 5.6]{dBM2}, holding for $2-$ and $3-$ forms on $2-$step nilpotent Lie groups.

In the nilpotent case, we show that the only parallel forms on a nilmanifold are linear combinations of wedge products of volume forms of some of the irreducible de Rham factors and of any left-invariant form on the flat factor (Corollary \ref{deK}). Therefore,  in order to understand left-invariant Killing forms on nilmanifolds, it is enough to study the de Rham irreducible case.

The algebraic system which translates the Killing condition at the Lie algebra level is very involved for $k\ge 4$, even on $2-$step nilmanifolds. We solved the problem completely only when $(N,g)$ is a 2-step nilpotent Lie group with left-invariant metric, whose Lie algebra $\mn$ has at most two-dimensional center. 

More precisely, when $\mn$ is non-abelian and its center is one-dimensional, $\mn$ is isomorphic to the Heisenberg Lie algebra $\mh_{2n+1}$ and its metric structure can be described by an invertible matrix in $\mathfrak{so}(2n)$. We show in Theorem \ref{aa} below that for any such matrix, the space of left-invariant Killing $k-$forms on the Heisenberg group is one-dimensional for every $k$ odd, and zero for $k$ even. The proof relies on a result of linear algebra of independent interest, which states that if $\omega$ is a non-degenerate 2-form on a finite-dimensional vector space $V$, then the only exterior forms $\gamma$ on $V$ satisfying 
$$(x\lrcorner\ \omega)\wedge (x\lrcorner \ \gamma)=0,\quad \mbox{ for every  }x\in V$$
are the constant multiples of the exterior powers of $\omega$.

Finally, in Theorem \ref{k} we show that if $(\mn,g)$ is an irreducible $2-$step nilpotent metric Lie algebra with two-dimensional center, then every left-invariant Killing $k-$form on the associated simply connected Riemannian Lie group $(N,g)$ vanishes if $4\le k\le \dim(\mn)-1$. The proof relies again on an argument of linear algebra (Proposition \ref{p63} below), whose proof, however, is rather involved. We suspect that when the dimension of the center increases, the corresponding linear algebra statements become intractable with standard methods.

\section{Riemannian geometry of Lie groups with left-invariant metrics}

Let $N$ be a connected Lie group endowed with a left-invariant Riemannian metric $g$, and let $\mn$ denote the Lie algebra of $N$, which we identify with the tangent space $T_eN$ of $N$ at the identity $e$. Left translations by elements of the Lie group are isometries, so the metric $g$ is determined by its value on $\mn$, which will also be denoted by $g$. 

Let $\nabla$ denote the Levi-Civita connection of $(N,g)$. Koszul's formula evaluated on left-invariant vector fields $X,Y,Z$ on $N$ reads
\begin{equation}\label{eq.Koszul}
\lela \nabla_X Y,Z\rira=\frac12\{ \lela [X,Y],Z\rira+\lela [Z,X],Y\rira+\lela [Z,Y],X\rira\}.
\end{equation}

This formula shows in particular that the covariant derivative of a left-invariant vector field with respect to another left-invariant vector field is again left-invariant. This enables one to define a linear map $\nabla:\mn\to\End(\mn)$ by $\nabla_xy:=(\nabla_XY)_e$, where $X$ and $Y$ are the left-invariant vector fields on $N$ whose values at $e$ are $x$ and $y$ respectively.

From now on we will identify a left-invariant vector field $X$ with its value $x\in\mn$ at the identity, so that \eqref{eq.Koszul} becomes 
\begin{equation}\label{eq.Koszul1}
\nabla_xy=\frac12\left( [x,y]-\ad_x^*y-\ad_y^*x\right),
\end{equation}
where $\ad_x^*$ denotes the adjoint of $\ad_x$ with respect to the inner product $g$ on $\mn$.

The center and the commutator of a Lie algebra $\mn$ are, respectively,
\[\mz=\{z\in\mn\ |\   [x,z]=0, \,\mbox{for all }x\in\mn\},\qquad
\mn'=[\mn,\mn]:=\mathrm{span}\{ [x,y]\ |\   x,y\in\mn\}.
\]

Assume now that $N$ is simply connected and let 
$$(N,\bil)=(N_0,\bil_0)\times (N_1,\bil_1)\times (N_2,\bil_2)\times \ldots \times (N_q,\bil_q),$$ 
be its de Rham decomposition through the identity $e\in N$, where $(N_0,\bil_0)$ is the Euclidean factor (possibly trivial) and $(N_i,\bil_i)$ are irreducible Riemannian manifolds for $1\le i\le q$. Each $N_i$ determines a subspace $\mn_i$ of the Lie algebra $\mn$ of $N$,  namely $\mn_i:=T_eN_i$. Then we have $\mn=\mn_0\oplus \mn_1\oplus \ldots \oplus\mn_q$ as an orthogonal direct sum of vector spaces.

\begin{pro}\label{pro:suba}
For each $i=0,\ldots, q$, $\mn_i$ is a subalgebra of $\mn$ and $\mn_0$ is abelian. Moreover, any direct sum $\bigoplus_{j=1}^s \mn_{i_j}$, with $i_j\in \{0, \ldots, q\}$ is also a subalgebra of $\mn$.
\end{pro}

\begin{proof}
Let $f$ be an isometry in the connected component of the identity of the isometry group of $(N,\bil)$. Then $f$ preserves the parallel distributions determining $N_i$ for each $i=0, \ldots, q$, that is, $\mathrm{d}f_e(T_eN_i)=T_{f(e)}N_i$ \cite[Theorem 3.5, Ch. VI]{KO-NO}. 

Taking $f=L_a$ to be a left-translation, we get
\begin{equation}\label{eq:Ti}
T_aN_i=\mathrm{d}(L_a)_e\mn_i,\quad \mbox{ for every }a\in N, \; i=0, \ldots, q.
\end{equation}

Let $x,y\in \mn_i$ and denote by $X,Y$ the left-invariant vector fields on $N$ that they induce. Then by \eqref{eq:Ti} $X,Y\in TN_i$ and thus $[X,Y]\in TN_i$. Hence $[X,Y]_e\in \mn_i$ and $\mn_i$ is a subalgebra. Similarly, we obtain the result for an arbitrary sum of  $\mn_i$.
\end{proof}

Notice that in general it is not true that the sum of two subalgebras of a Lie algebra is again a Lie subalgebra.

\begin{pro}\label{pro:idn}
If the Lie group $N$ is nilpotent, then $\mn_i$ is an ideal of $\mn$ for every $i=0,\ldots, q$.
\end{pro}

\begin{proof}
For $q=0$ there is nothing to show. Assume $q\ge 1$ and consider the orthogonal decomposition of $\mn$ in a sum of subalgebras $\mn=\mh_1\oplus  \mh_2$, where $\mh_1:=\mn_0\oplus \mn_1\oplus \ldots\oplus \mn_{q-1}$ and $\mh_2:=\mn_q$. 

From the definition of $\mn_i$, the Levi-Civita connection $\nabla$ preserves each $\mn_i$, that is, $\nabla \mn_i\subset \mn_i$, for every $i=0, \ldots, q$, so in particular $\nabla \mh_i\subset \mh_i$, for $i=1,2$.

Let $x,y\in \mh_1$ and $z\in \mh_2$. Then by \eqref{eq.Koszul}, 
$$\lela [z,x],y\rira +\lela [z,y],x\rira=2\lela \nabla_x y,z\rira-\lela [x,y],z\rira=0,  $$ since $\nabla_x y$ and $[x,y]$ are in $\mh_1$, which is orthogonal to $\mh_2$. Therefore, in an orthonormal basis adapted to the decomposition $\mh_1\oplus \mh_2$, $\ad_z$  has the following form:
$$\ad_z=\left(\begin{matrix}A&0\\\star &B\end{matrix}\right),$$
where $A$ is skew-symmetric.
Since $N$ is nilpotent, $\ad_z$ is a nilpotent endomorphism of $\mn$ for every $z\in\mn$, so $A$ and $B$ are both nilpotent matrices. Hence $A=0$ and, since $z$ is arbitrary in $\mh_2$, $[\mh_1,\mh_2]\subset \mh_2$. A similar argument gives $[\mh_1,\mh_2]\subset \mh_1$ so finally we obtain $[\mh_1,\mh_2]=0$.

Thus $\mn_q$ and $\mn_0\oplus \mn_1\oplus \ldots\oplus \mn_{q-1}$ are ideals of $\mn$, and the statement follows by immediate induction.
\end{proof}

\begin{ex}
If the Lie group $N$ is not assumed to be nilpotent, the tangent spaces at the identity of the de Rham factors are not ideals of the Lie algebra of $N$ in general. For example, let $(\mh,\bil)$ be a de Rham irreducible metric Lie algebra and let $A\in \so(\mh)\cap \Der(\mh)$ be a skew-symmetric derivation. Consider the Lie algebra $\mn:=\R\xi \ltimes_A \mh$ together with an inner product extending $\bil$ and such that $\xi\bot \mh$. From Koszul's formula \eqref{eq.Koszul}, we obtain that $\nabla_\xi \xi=0$ and $\nabla\mh\subset \mh$. Nevertheless, $\R\xi$ is not an ideal of $\mn$ if $A\neq 0$.
\end{ex}

A Lie algebra endowed with an inner product $(\mn,\bil)$ is called {\em reducible} if it can be written as an orthogonal direct sum of ideals $\mn=\mn_1\oplus\mn_2$. In this case we endow $\mn_i$ with the inner product $\bil_i$ which is the restriction of $\bil$ to $\mn_i$, for each $i=1,2$. Otherwise, $(\mn,g)$ is called {\em irreducible}. Proposition \ref{pro:idn} states that simply connected de Rham irreducible nilpotent Lie groups are in correspondence with irreducible nilpotent metric Lie algebras.

Summarizing, we get:

\begin{teo}\label{teo:dR}
Let $(N,\bil)$ be a connected and simply connected nilpotent Lie group endowed  with a left-invariant Riemannian metric, and consider its de Rham decomposition
$$(N,\bil)=(N_0,\bil_0)\times (N_1,\bil_1)\times (N_2,\bil_2)\times \ldots \times (N_q,\bil_q).$$
 Then each $(N_i,\bil_i)$, with $i=1,\ldots, q$, is (isometric to) a connected, simply connected irreducible nilpotent Lie group endowed with a left-invariant metric. In particular, the Lie algebra $\mn$  of $N$ is a direct sum of orthogonal ideals 
 $$(\mn,\bil) =(\mn_0,\bil_0)\oplus \bigoplus_{i=1}^q (\mn_i,\bil_i), $$
 where $\mn_0$ is abelian and $\mn_i$ is nilpotent, non-abelian and irreducible for $i=1, \ldots, q$.
\end{teo}

This result (which is a direct consequence of Propositions \ref{pro:suba} and \ref{pro:idn}) generalizes Corollary A.4 in \cite{dBM} (proved for 2-step nilpotent Lie groups) to nilpotent Lie groups of any nilpotency degree.

\section{Left-invariant Killing forms on Lie groups}

In this section we introduce the concept of Killing forms on Riemannian manifolds. We focus on such forms defined on Lie groups endowed with left-invariant metrics, and we study their behavior with respect to the de Rham decomposition. 

\begin{defi}
A Killing $k-$form on a Riemannian manifold $(M,g)$ is a differential $k-$form $\alpha$ that satisfies 
\begin{equation}\label{eq:Killform}
\nabla_X \alpha=\frac1{k+1}X\lrcorner\ \mathrm{d}\alpha
\end{equation}
for every vector field $X$ in $M$, where $\lrcorner$ denotes the contraction (or interior product) of differential forms by vector fields.	
\end{defi}
Equivalently, $\alpha$ is a Killing form if and only if  $X\lrcorner\  \nabla_X\alpha=0$ for every vector field $X$ (see \cite{Se03}). If $\xi$ is a vector field on $M$ and $\alpha$ is its metric dual $1-$form, i.e. $\alpha=\lela \xi, \cdot\rira$, then $\alpha$ is a Killing $1-$form if and only if $\xi$ is a Killing vector field. 
\smallskip

If $N$ is a Lie group with Lie algebra $\mn$ and $g$ is a left-invariant metric on $N$, every left-invariant differential $k-$form  $\alpha$ on $N$ is determined by its value at the identity; hence one can identify  left-invariant $k-$forms on $N$ with elements in $\Lambda^k\mn^*$. The covariant and exterior derivatives preserve left-invariance, thus $\alpha\in \Lambda^k\mn^*$ defines a left-invariant Killing form on $N$ if and only if 
\begin{equation}\label{eq:Killforminv}
\nabla_y \alpha=\frac1{k+1}\ y\lrcorner\ \mathrm{d}\alpha, \quad \mbox{ for all } y\in \mn,
\end{equation}
where $\mathrm{d}:\Lambda^k\mn^*\lra \Lambda^{k+1}\mn^*$ is the Lie algebra differential.  As mentioned before, \eqref{eq:Killforminv} is  equivalent to $y\lrcorner\  \nabla_y\alpha=0$ for all $y\in \mn$. 

From now on we will identify $\mn$ with $\mn^*$ using the metric and denote by $\mathcal{K}^k(\mn,g)\subset\Lambda^k\mn$ the space of Killing $k-$forms on $(\mn,g)$, i.e. elements $\alpha\in\Lambda^k\mn$ satisfying \eqref{eq:Killforminv}.

We shall first study Killing forms on Riemannian products of Lie groups endowed with left-invariant metrics. Let $(N_1,\bil_1)$, $(N_2,\bil_2)$ be Lie groups endowed with a left-invariant metric and consider $N=N_1\times N_2$ endowed with the product metric. Then $\mn=\mn_1\oplus \mn_2$ is an orthogonal direct sum of ideals.

Reducibility of $\mn$ induces a decomposition of the space of $k-$forms so that any $\alpha\in \Lambda^k\mn$ can be written as $\alpha=\sum_{l=0}^k\alpha_l$, where $\alpha_l\in \Lambda^l\mn_1 \otimes \Lambda^{k-l}\mn_2$, for $l=0, \ldots, k$.

The following is a Lie group analogue of a result proved in \cite{ms08} about the decomposition of Killing forms on a product of compact Riemannian manifolds.
\begin{pro}\label{p32}
A left invariant $k-$form $\alpha$ on $N$ is a Killing form on $(N,g)$ if and only if $\alpha_0$ and $\alpha_k$ are Killing forms on $N_2$ and $N_1$ respectively, and $\alpha_l$ is a parallel form on $N$ for each $l=1, \ldots, k-1$.
\end{pro}
\begin{proof} 
Let $\alpha$ be a left-invariant $k-$form on $N$ and write $\alpha=\sum_{l=0}^k \alpha_l$, with $\alpha_l\in \Lambda^l\mn_1 \otimes \Lambda^{k-l}\mn_2$. Since $\nabla \mn_1\subset \mn_1$ and $\nabla \mn_2\subset \mn_2$, we have that $\nabla_x\alpha_l\in \Lambda^l\mn_1 \otimes \Lambda^{k-l}\mn_2$ for every $x\in \mn$ and $l=0, \ldots, k$. 

Suppose now that $\alpha$ is a Killing form and let $x\in \mn_1$. Then, by \eqref{eq:Killforminv}, it satisfies
\begin{equation}
\label{eq:sum}
0=x\lrcorner\   \nabla_x\alpha=\sum_{l=0}^kx\lrcorner\  \nabla_x \alpha_l.
\end{equation}

For every $x\in\mn_1$, $x\lrcorner\ \nabla_x\alpha_l\in \Lambda^{l-1}\mn_1 \otimes \Lambda^{k-l}\mn_2$ so each term in \eqref{eq:sum} lies in a different summand of $\Lambda^k(\mn_1\oplus \mn_2)$.  Therefore 
\begin{equation}\label{xx}x\lrcorner\  \nabla_x \alpha_l=0\qquad \hbox{for every }x\in \mn_1\hbox{ and }l=0, \ldots, k.\end{equation}
Similarly, we get 
\begin{equation}\label{yy}y\lrcorner\  \nabla_y \alpha_l=0\qquad \hbox{for every }y\in \mn_2\hbox{ and }l=0, \ldots, k.
\end{equation}

In particular, $x\lrcorner\ \nabla_x\alpha_k=0$ for all $x\in \mn_1$, so $\alpha_k$, which is an element of $\Lambda^k\mn_1$, defines a Killing form on $(N_1,g_1)$. Using a similar argument we obtain that $\alpha_0$ defines a Killing form on $(N_2,g_2)$. 

Consider orthonormal bases $\{u_1, \ldots, u_n\}$ of $\mn_1$ and $\{v_1, \ldots, v_m\}$ of $\mn_2$. We define the operators $\mathrm{d}_1,\mathrm{d}_2:\Lambda^*\mn\to\Lambda^{*+1}\mn$
by
$$\mathrm{d}_1:=\sum_{j=1}^n u_j\wedge \nabla_{u_j},
\quad\mathrm{d}_2:=\sum_{j=1}^m v_j\wedge \nabla_{v_j}
.$$
Of course, these operators do not depend on the chosen orthonormal bases, and it is straightforward to check the relations 
\begin{equation} \mathrm{d}_i\mathrm{d}_j+\mathrm{d}_j\mathrm{d}_i=0,
\qquad\hbox{for all } i,j\in\{1,2\}.
\label{dd}
\end{equation}

From \eqref{xx} and \eqref{yy} we immediately get for every $x\in \mn_1,\ y\in \mn_2$ and $ l=0,\ldots, k$:
\begin{equation}\label{dd1}
\nabla_x \alpha_l=\frac1{l+1}x\lrcorner\ \dd_1\alpha_l\quad \mbox{ and } \quad \nabla_y \alpha_l=\frac1{k-l+1}y\lrcorner\ \dd_2\alpha_l.
\end{equation}

Moreover, since $\alpha$ is a Killing form, we have $ x\lrcorner\  \nabla_y\alpha +y\lrcorner\  \nabla_x\alpha=0,$ for every $x\in \mn_1$, $y\in \mn_2$. Projecting this equality onto the different summands of $\Lambda^{k-1}(\mn_1\oplus\mn_2)$ we obtain
\begin{equation}\label{eq:nxy}
 x\lrcorner\ \nabla_y\alpha_l +y\lrcorner\ \nabla_x\alpha_{l-1}=0,\quad  \mbox{for } l=0,\ldots, k
 \end{equation}
(where by convention $\alpha_{-1}=0$).
We then obtain
\begin{equation}\label{eq:dn}
y\lrcorner\ \mathrm{d}_1\alpha_{l-1}= y\lrcorner\  \left(\sum_{j=1}^n u_j\wedge \nabla_{u_j}\alpha_{l-1}\right)=\sum_{j=1}^n u_j\wedge ( u_j\lrcorner\ \nabla_y\alpha_l) =l\nabla_y\alpha_l,
\end{equation}
for every $y\in \mn_2$ and $ l=0,\ldots, k$.

Taking the wedge product with $y$ in \eqref{eq:dn} and summing over $y=v_j$ yields
\begin{equation}\label{eq:d2oi}
l\,\mathrm{d}_2\alpha_l=l\sum_{j=1}^m v_j\wedge \nabla_{v_j} \alpha_l=\sum_{j=1}^m v_j\wedge( {v_j} \lrcorner\ \mathrm{d}_1\alpha_{l-1})=(k-l+1)\mathrm{d}_1\alpha_{l-1},
\end{equation}
for every $l=0,\ldots, k$. Applying $\dd_1$ to this relation and using \eqref{dd} for $i=j=1$ we get
\begin{equation}\label{eq:d3}
\dd_1\mathrm{d}_2\alpha_l=0, \quad  \mbox{for } l=1,\ldots, k,
\end{equation}
(which by \eqref{dd} also holds for $l=0$ since $\dd_1\alpha_0=0$).
%
%


For every $y\in \mn_2$ and $l=1, \ldots, k$ we can write, using \eqref{dd1}, \eqref{eq:d2oi} and \eqref{eq:d3}:
\begin{eqnarray*}\nabla_y(\dd_2\alpha_l)&=&\frac{k-l+1}{l}\nabla_y(\dd_1\alpha_{l-1})=\frac{k-l+1}{l}\dd_1(\nabla_y\alpha_{l-1})\\
&=&\frac{k-l+1}{l(k-l+2)}\dd_1(y\lrcorner\ \dd_2\alpha_{l-1})
=-\frac{k-l+1}{l(k-l+2)}y\lrcorner\ (\dd_1\dd_2\alpha_{l-1})=0.
\end{eqnarray*}

Consequently, for every $y\in\mn_2$ and $l=1, \ldots, k$ we have
$$\nabla_y(\nabla_y\alpha_l)=\frac{1}{k-l+1}\nabla_y(y\lrcorner\ \dd_2\alpha_l)=\frac{1}{k-l+1}(\nabla_yy)\lrcorner\ \dd_2\alpha_l=\nabla_{\nabla_yy}\alpha_l.$$
Taking the scalar product with $\alpha_l$ and using the skew-symmetry of the operators $\nabla_y$ and $\nabla_{\nabla_yy}$ acting on exterior forms, we finally get
$$|\nabla_y\alpha_l|^2=-g(\nabla_y(\nabla_y\alpha_l),\alpha_l)=-g(\nabla_{\nabla_yy}\alpha_l,\alpha_l)=0,$$
so $\nabla_y\alpha_l=0$ for every $y\in \mn_2$ and $l=1, \ldots,k$. By \eqref{eq:nxy}, we also get $\nabla_x\alpha_l=0$ for every $x\in \mn_1$ and $l=0, \ldots, k-1$.

Consequently, we proved that $\nabla_z\alpha_l=0 $ for every $z\in \mn$ and every $l=1, \ldots, k-1$. Hence, $\alpha_l$ defines a parallel form on $(N,g)$, for each $l=1, \ldots, k-1$.
\end{proof}

In order to understand left-invariant Killing forms on products of nilpotent Riemannian Lie groups, it is thus necessary to study left-invariant parallel forms on such manifolds. 

Since parallel forms correspond to fixed points of the holonomy representation on the exterior bundle, it is straightforward to check that every parallel form on a product of Riemannian manifolds is a sum of wedge products of parallel forms on the factors of this product. 

The next result shows that de Rham irreducible  nilpotent Lie groups do not carry any parallel forms besides the obvious ones.

\begin{pro}\label{pro:np}
Let $(N,g)$ be a de Rham irreducible connected and simply connected (non-abelian) nilpotent Lie group endowed with a left-invariant metric. Then the only parallel differential forms on $(N,g)$ are the constants and the constant multiples of the volume form. In particular, $(N,g)$ does not admit K\"ahler structures.
\end{pro}

\begin{proof}
First, recall that nilpotent Lie groups endowed with left-invariant metrics are never Einstein \cite{Mil76}.
In particular, $(N,\bil)$ is not an irreducible symmetric space. 

If $(N,g)$ carries a non-zero left-invariant parallel form of degree $k$ with $0<k<\dim(N)$, then it has special holonomy. Indeed, for $1\le k\le n-1$, the representation of ${\rm SO}(n)$ in $\Lambda^k\mathbb{R}^n$ has no fixed points. Since $N$ is not Einstein, the Berger-Simons holonomy theorem implies that $(N,\bil)$ is a K\"ahler manifold with holonomy ${\rm U}(m)$, where $m=\frac12\dim(N)$ \cite[Ch. 10]{Be86}.	

We shall prove that the K\"ahler form $\omega$ of $(N,g)$ is necessarily left-invariant. This will lead us to a contradiction, since nilpotent Lie groups endowed with left-invariant metrics do not carry left-invariant K\"ahler structures \cite{AD}. 

Since $(N,g)$ has holonomy ${\rm U}(m)$, any parallel form on $(N,\bil)$ has even degree and is a multiple of $\omega^s$, for some $s\in \{0,\ldots, m\}$. In particular, since left-translations are isometries, for every $a\in N$ there exists $\lambda(a)\in \R$ such that $L_a^*\omega=\lambda(a)\omega$. This implies that the volume form $\omega^{m}$ verifies $L_a^*\omega^{m}=\lambda(a)^m\omega^{m}$. But since the metric is left-invariant and $N$ is connected, the volume form is preserved by left-translations, that is  $L_a^*\omega^{m}=\omega^{m}$. Therefore, $\lambda(a)=1$ for all $a\in N$, or equivalently, $\omega$ is left-invariant as claimed.
\end{proof}

\begin{remark} 
Let $N$ be a simply connected (non-abelian) nilpotent Lie group.
The last statement of Proposition \ref{pro:np} says that $N$ carries no Kähler structure $(g,J)$ with left-invariant underlying metric $g$. One can compare this with a result in \cite{BE-GO}, where the authors prove that $N$ does not admit K\"ahler structures $(g,J)$ with both $g$ and $J$ invariant with respect to some co-compact lattice. 
\end{remark}

Theorem \ref{teo:dR} and Propositions \ref{p32} and \ref{pro:np} give the following decomposition result for left-invariant Killing forms on nilpotent Lie groups.

\begin{cor}\label{deK}
Every left-invariant Killing form on a connected and simply connected nilpotent Lie group is the sum of Killing forms on its de Rham factors, and a left-invariant parallel form. The latter is a linear combination of wedge products of volume forms of some of the irreducible de Rham factors and of any left-invariant form on the flat factor.
\end{cor}

\section{Left-invariant Killing forms on irreducible 2-step nilpotent Lie groups}

In this section  we recall some basic facts about the geometry of 2-step nilpotent Lie groups endowed with a left-invariant metric and their left-invariant Killing forms; we refer to \cite{dBM2} for further details. By Corollary \ref{deK}, the study of left-invariant Killing forms on these manifolds reduces to considering only the de Rham irreducible case.

Let $\mn$ be a non-abelian $2-$step nilpotent Lie algebra, i.e. $0\neq \mn'\subseteq \mz$, and let $N$ be its corresponding connected and simply connected $2-$step nilpotent Lie group.  We consider a left-invariant metric $g$ on $N$ so that $(N,g)$ is  a Riemannian manifold, which we assume de Rham irreducible. By Theorem \ref{teo:dR}, $(\mn,g)$ is irreducible as metric Lie algebra and, in particular, $\mz=\mn'$.

The main geometric properties of $(N,g)$ can be described through objects in the metric Lie algebra $(\mn,g)$ (see \cite{EB}).

Let $\mv$ denote the orthogonal complement of $\mz$ in $\mn$ so that $\mn=\mv\oplus\mz$ as an orthogonal direct sum of vector spaces.  Each central element $z\in\mz$ defines an endomorphism $j(z):\mv\lra\mv$ by the equation
\begin{equation}\label{eq:jota}
\lela j(z)x,y\rira =\lela z,[x,y]\rira, \quad \mbox{ for all } x,y\in\mv.
\end{equation}

Let $\sso(\mv)$ denote the Lie algebra of skew-symmetric endomorphisms of $\mv$ with respect to $\bil$. It is straightforward that $j(z)\in \so(\mv)$ for all $z\in \mz$.
 
The covariant derivative of left-invariant vector fields can be expressed using this map. Indeed, using \eqref{eq.Koszul1} we readily obtain 
\begin{equation}\label{eq:nabla}\left\{
\begin{array}{ll}
\nabla_x y=\frac12 \,[x,y] & \mbox{ if } x,y\in\mv,\\
\nabla_x z=\nabla_zx=-\frac12 j(z)x & \mbox{ if } x\in\mv,\,z\in\mz,\\
\nabla_z z'=0& \mbox{ if } z, z'\in\mz.
\end{array}\right.
\end{equation}

The irreducibility of $\mn$ implies that the linear map $j:\mz\lra \so(\mv)$ is injective. We recall the following general fact that will be used later in the presentation. A proof can be found in \cite{dBM2}.

\begin{lm} \label{lm:imjzv} If $\mn=\mv\oplus\mz$ is a $2-$step nilpotent Lie algebra, then $\mv=\sum_{z\in\mz}{\rm Im} j(z)$. Equivalently, $\bigcap_{z\in \mz}\ker j(z)=\{0\}$.
\end{lm}

In what follows we recall some properties of left-invariant  Killing forms on $2-$step nilpotent Lie groups. We include improved versions of the results in  \cite{dBM2}.

The space of $k-$forms on $\mn$ can be decomposed, using the fact that  $\mn=\mv\oplus\mz$, as follows
\begin{equation}\label{eq:lambdad}
\Lambda^k\mn=\bigoplus_{l=0}^k \Lambda^l\mv\otimes \Lambda^{k-l}\mz.
\end{equation}
Given $\alpha\in \Lambda^k\mn$, we write accordingly $\alpha_l$ for the projection of $\alpha$ on $\Lambda^l\mv\otimes \Lambda^{k-l}\mz$ with respect to this decomposition. Recall that $\mn$ is assumed to be irreducible so the decomposition of $k-$forms in \eqref{eq:lambdad} is not the same that the one considered in the previous section for reducible Lie algebras.

In order to state the characterization of Killing forms, we recall that every skew-symmetric endomorphism $A$ of an inner product space $(V,\bil)$ extends as a derivation, denoted by $A_*$,  of the exterior algebra $\Lambda^*V$ by the formula 
\begin{equation}\label{der} A_*\alpha:=\sum_{i=1}^nAe_i\wedge e_i\lrcorner\ \alpha,\qquad \mbox{ for all }\alpha\in \Lambda^*V,
\end{equation}
where $\{e_1, \ldots, e_n\}$ is an orthonormal basis of $(V,\bil)$.  It is straightforward that $A_*$ preserves the degree, verifies the Leibniz rule with respect to the interior and wedge products and, if $\alpha\in \Lambda^2V$ is viewed as a skew-symmetric endomorphism of $V$, then $A_*\alpha=[A,\alpha]$. Moreover, the map from $\so(V)$ to $\so(\Lambda^kV)$ defined by $A\mapsto A_*$ is a Lie algebra endomorphism.

\begin{pro} \label{pro:killgen} Let $\mn$ be a $2-$step nilpotent Lie algebra with orthogonal decomposition $\mn=\mv\oplus\mz$, and consider  orthonormal bases $\{e_1, \ldots, e_n\}$ and $\{z_1, \ldots, z_m\}$ of $\mv$ and $\mz$, respectively. 
Then a $k-$form $\alpha$ on $\mn$ defines a left-invariant Killing form on $(N,g)$ if and only if for every $x\in \mv$ and $z\in \mz$,  its components with respect to \eqref{eq:lambdad} verify 
\begin{equation}
\sum_{t=1}^mj(z_t)x\wedge (x\lrcorner\ z_t\lrcorner\ \alpha_{k-1})=0,\label{eq:i1}\end{equation}
 and 
 \begin{equation}
\sum_{i=1}^n [x,e_i]\wedge (z\lrcorner \ e_i\lrcorner\ \alpha_{l+1})=j(z)x\lrcorner\ \alpha_{l+1} -x\lrcorner\  j(z)_*\alpha_{l+1}
+ \sum_{t=1}^m j(z_t)x\wedge (z\lrcorner \ z_t\lrcorner\ \alpha_{l-1}),\label{eq:i2}
\end{equation}
for every $l=0,\ldots,k-1$.
\end{pro}

\begin{proof} From \cite[Corollary 3.3]{dBM2} we know that $\alpha\in \Lambda^k\mn$ defines a left-invariant Killing form on $(N,g)$ if and only if  for every $l=0,\ldots,k-1$, $x\in \mv$ and $z\in \mz$ the following equations hold:
\begin{eqnarray}
\sum_{i=1}^n [x,e_i]\wedge (x\lrcorner \ e_i\lrcorner\ \alpha_{l+2}) &=& \sum_{t=1}^mj(z_t)x\wedge (x\lrcorner \ z_t\lrcorner\ \alpha_l),\label{eq:pp1}\\
 \sum_{i=1}^n j(z)e_i\wedge (z\lrcorner \ e_i\lrcorner\ \alpha_l) &=&0,\label{eq:pp2}\\
\sum_{i=1}^n [x,e_i]\wedge (z\lrcorner \ e_i\lrcorner\ \alpha_{l+1})&=&2( j(z)x)\lrcorner\ \alpha_{l+1} + \sum_{i=1}^n j(z)e_i\wedge (x\lrcorner \ e_i\lrcorner\ \alpha_{l+1})\label{eq:pp3}\\
&&\qquad+ \sum_{t=1}^m j(z_t)x\wedge (z\lrcorner \ z_t\lrcorner\ \alpha_{l-1}). \nonumber
\end{eqnarray}

From  \eqref{der}, it is easy to check that \eqref{eq:pp3} is equivalent to \eqref{eq:i2}. Moreover, \eqref{eq:i1} is just \eqref{eq:pp1} for $l=k-1$. It remains to show that \eqref{eq:i2} implies \eqref{eq:pp2} for $l=0,\ldots,k-1$ and \eqref{eq:pp1} for $l=0,\ldots,k-2$.

By \eqref{der}, together with the fact that $j(z)z=0$, \eqref{eq:pp2} is equivalent to $z\lrcorner\ j(z)_*\alpha_l=0$ for $l=0, \ldots, k-1$. Taking the interior product with $z$ in \eqref{eq:i2} and using \eqref{eq:jota}, we obtain $z\lrcorner\ x \lrcorner\ j(z)_*\alpha_{l+1}=0$ for every $x\in \mv$. Taking $x=e_i$ in this expression, making the wedge product with $e_i$ and summing over $i$ we get $z \lrcorner\ j(z)_*\alpha_{l+1}=0$ for all $z\in \mz$ and $l=0,\ldots,k-1$. 
Thus, \eqref{eq:i2} implies \eqref{eq:pp2} for $l=1, \ldots, k-1$. Moreover, \eqref{eq:pp2} for $l=0$ is trivially satisfied.

We now take the interior product with $x$ in \eqref{eq:i2} and obtain
\begin{eqnarray}
\sum_{i=1}^n [x,e_i]\wedge (x\lrcorner\ z\lrcorner \ e_i\lrcorner\ \alpha_{l+1})&=&x\lrcorner\ j(z)x\lrcorner\ \alpha_{l+1} - \sum_{t=1}^m j(z_t)x\wedge (x\lrcorner\ z\lrcorner \ z_t\lrcorner\ \alpha_{l-1}). \nonumber
\end{eqnarray}
In this equation, put $z=z_s$, take the wedge product with  $z_s$ and sum over $s=1, \ldots, m$, to obtain
\begin{eqnarray*}
 -(k-l-1)
 \sum_{i=1}^n [x,e_i]\wedge  ( x\lrcorner\ e_i\lrcorner\ \alpha_{l+1})&=& \sum_{i=1}^n [x,e_i]\wedge (x\lrcorner\ e_i\lrcorner\ \alpha_{l+1}) 
 \\
 &&\quad-(k-l) \sum_{t=1}^m  j(z_t)x\wedge( x\lrcorner\   z_t\lrcorner\ \alpha_{l-1}),
\end{eqnarray*}
which reads 
$$  \sum_{i=1}^n [x,e_i]\wedge  ( x\lrcorner\ e_i\lrcorner\ \alpha_{l+1})= \sum_{t=1}^m  j(z_t)x\wedge( x\lrcorner\   z_t\lrcorner\ \alpha_{l-1}).$$
This last equation is nothing but \eqref{eq:pp1} for $l-1$. We thus have proved that \eqref{eq:i2} for $l=0, \ldots, k-1$ implies \eqref{eq:pp1} for $l=0, \ldots, k-2$.
\end{proof}

\begin{remark}\label{rem:evodd} 
Let $\alpha$ be a $k-$form on $\mn$ and consider its components $\alpha_l$ with respect to the decomposition in \eqref{eq:lambdad}. Denote 
\begin{equation*}
\alpha^{\rm ev}:=\sum_{l:\, k-l \mbox{\tiny{ even}}}\alpha_l=\alpha_k+\alpha_{k-2}+\ldots\quad \mbox{ and } \quad \alpha^{\rm odd}:=\sum_{l:\, k-l\mbox{\tiny{ odd}}}\alpha_l=\alpha_{k-1}+\alpha_{k-3}+\ldots.
\end{equation*}
The system of equations in Proposition \ref{pro:killgen} splits into two uncoupled systems: one involving $\alpha^{\rm ev}$ and another one involving $\alpha^{\rm odd}$. Therefore, $\alpha$ defines a Killing form on $(N,g)$ if and only if both $\alpha^{\rm ev}$ and $\alpha^{\rm odd}$ define Killing forms on $(N,g)$.
\end{remark}

\begin{pro}\label{ref:jzl}
Let $\alpha$ be a $k-$form  on a $2-$step nilpotent Lie algebra $\mn=\mv\oplus\mz$ defining a left-invariant Killing form on $(N,g)$ and consider $\{z_1, \ldots,z_m\}$ an orthonormal basis of the center. Then for every $x\in \mv$, $z\in \mz$,
\begin{eqnarray}\label{eq:i3}
j(z)_*\alpha_{l+1}=\frac2{k+1}\sum_{t=1}^m j(z_t)\wedge ( z \lrcorner\ z_t  \lrcorner\ \alpha_{l-1}),  \qquad l=0, \ldots, k-1.
\end{eqnarray}
Moreover, $\alpha_1=0$.
\end{pro}
\begin{proof} Let $\alpha\in \Lambda^k\mn$ define a Killing form on $(N,g)$ and write $\alpha=\sum_{l=0}^k\alpha_l$ with respect to the decomposition in \eqref{eq:lambdad}. Consider an orthonormal basis $\{e_1, \ldots, e_n\}$ of $\mv$.  According to Proposition \ref{pro:killgen}, $\alpha_l$ satisfies \eqref{eq:i2} for each $x=e_i$,  $i=1, \ldots, n$ and $z\in \mz$. Taking the interior product  of this equation with $e_i$ and summing over $i$ we obtain
$$-\sum_{t=1}^m z_t\wedge j(z_t)_*(z\lrcorner \ \alpha_{l+1})=-(l+2)j(z)_*\alpha_{l+1}+2\sum_{t=1}^mj(z_t)\wedge (z\lrcorner\ z_t\lrcorner\ \alpha_{l-1}).$$
Therefore, 
\begin{eqnarray}
z\ \lrcorner\ \left(\sum_{t=1}^m z_t\wedge j(z_t)_*\alpha_{l+1}\right)&=&j(z)_*\alpha_{l+1}-\sum_{t=1}^m z_t\wedge j(z_t)_*(z \lrcorner \ \alpha_{l+1})\nonumber\\
&=&-(l+1)j(z)_*\alpha_{l+1}+2\sum_{t=1}^mj(z_t)\wedge (z\lrcorner\ z_t\lrcorner\ \alpha_{l-1}).\label{eq:zin}
\end{eqnarray}
Putting $z=z_s$ in this equation, taking the interior product with $z_s$ and summing over $s$, we obtain
\begin{eqnarray}
\sum_{t=1}^m z_t\wedge j(z_t)_*\alpha_{l+1}
&=&-\frac{l+1}{k-l}\sum_{t=1}^m z_t\wedge j(z_t)_*\alpha_{l+1}+2\sum_{t=1}^mj(z_t)\wedge (z_t\lrcorner\ \alpha_{l-1}),\nonumber
\end{eqnarray}
which implies
$$
\sum_{t=1}^m z_t\wedge j(z_t)_*\alpha_{l+1} =\frac{2(k-l)}{k+1}\sum_{t=1}^mj(z_t)\wedge (z_t\lrcorner\ \alpha_{l-1}).
$$
Finally, we take the interior product with $z\in\mz$ in this equation so that we get \eqref{eq:i3}.

To prove that $\alpha_1=0$, first notice that \eqref{eq:i3} implies, for $l=0$, $j(z)_*\alpha_1=0$ and thus $j(z)_*(x\lrcorner\ \alpha_1)=j(z)x\lrcorner\ \alpha_1$. Moreover, since $x\lrcorner\ \alpha_1$ is a $k-1-$form on $\mz$, we also have $j(z)_*(x\lrcorner\ \alpha_1)=0$. Therefore, $j(z)x\lrcorner\ \alpha_1=0$ for every $z\in \mv$ and $x\in \mv$, which implies $\alpha_1=0$ in view of Lemma \ref{lm:imjzv}.
\end{proof}

\begin{cor}\label{cor:l1}
If $\alpha$ defines a Killing $k-$form and $\alpha_{l-1}=0$ for some $l\in \{1, \ldots, k-1\}$, then  $j(z)_*\alpha_{l+1}=0$ for every $z\in \mz$, and for every $x\in \mv$ one has 
\begin{equation}\label{eq:i2eq}\sum_{t=1}^m z_t\wedge  (j(z_t)x \lrcorner\ \alpha_{l+1})=0,
\end{equation}
where $\{z_1, \ldots, z_m\}$ is an orthonormal basis of $\mz$.
\end{cor}

\begin{proof} 
Assume that $\alpha$ defines a Killing $k-$form and  $\alpha_{l-1}=0$ for some $l\in \{1, \ldots,k-1\}$. Then, it is clear from \eqref{eq:i3} that $j(z)_*\alpha_{l+1}=0$. Moreover, \eqref{eq:i2} holds and, since $j(z)_*\alpha_{l+1}=0$, it becomes
\begin{equation}\label{eq:8}
\sum_{i=t}^n z_t\wedge (z\lrcorner \ j(z_t)x\lrcorner\ \alpha_{l+1})=j(z)x\lrcorner\ \alpha_{l+1}, \quad \mbox{ for any }x\in \mv, \;z\in \mz.
\end{equation}
The left hand side in this equation equals 
\[j(z)x\lrcorner\ \alpha_{l+1}-
z\lrcorner\  \left(\sum_{i=t}^n z_t\wedge (j(z_t)x\lrcorner\ \alpha_{l+1})\right)
\] so \eqref{eq:8} is equivalent to 
\begin{equation*}
z\lrcorner\  \left(\sum_{i=t}^n z_t\wedge ( j(z_t)x\lrcorner\ \alpha_{l+1})\right)=0,  \quad \mbox{ for any }x\in \mv, \;z\in \mz. 
\end{equation*}
Taking the wedge product with $z$ in this equation, replacing $z=z_s$ and summing over $s$ we obtain \eqref{eq:i2eq}.
\end{proof}

\section{Killing forms on $2-$step nilpotent Lie groups with one-dimensional center}

In this section we assume that the 2-step nilpotent Lie algebra $\mn$ is non-abelian and possesses a center of dimension one. It is easy to prove that in this situation, $\mn$ is isomorphic to the Heisenberg Lie algebra. Let $z_1$ be a norm one vector in the center and denote $A:=j(z_1)\in \so(\mv)$, so that $j(z)=\lela z,z_1\rira A$ for every $z\in \mz$. Note that $A$ is invertible by Lemma \ref{lm:imjzv}. In what follows we study the Killing condition on $k-$forms on $\mn$. 

According to \eqref{eq:lambdad}, the space of $k-$forms on $\mn$ splits as
\begin{equation}
\Lambda^k\mn= \Lambda^k\mv\oplus (\Lambda^{k-1}\mv\otimes \mz).
\end{equation}
Thus any $k-$form $\alpha$ on $\mn$ can be written as $\alpha=\alpha_k+\alpha_{k-1}=\beta+z_1\wedge\gamma$  where $\beta\in \Lambda^{k}\mv$ and $\gamma\in \Lambda^{k-1}\mv$. 

\begin{pro}\label{pro:Ag} In the notations above, $\alpha\in \Lambda^k\mn$ is a Killing form if and only if $\beta=0$ and
\begin{equation}\label{eq:Ag}
Ax \wedge (x\lrcorner \ \gamma)=0,\quad \mbox{ for every  }x\in \mv.
\end{equation}
\end{pro}

\begin{proof} Suppose $\alpha=\beta+z_1\wedge \gamma$ is a Killing $k-$form. Then, by Proposition \ref{pro:killgen}, \eqref{eq:i1} holds and, in the present notation, this equation is equivalent to  \eqref{eq:Ag}.

Equation \eqref{eq:i2} for $l=k-1$ reads
$$Ax\lrcorner \ \beta=x\lrcorner \ A_*\beta,$$ for $x\in \mv$. Moreover, since $\alpha_{k-2}=0$, we have that $A_*\beta=0$ from Corollary \ref{cor:l1}. Hence the equality above implies $Ax\lrcorner\ \beta=0$ for every $x\in \mv$, and since $A$  is surjective, we get $\beta=0$.

Conversely, suppose that $\alpha=z_1\wedge \gamma$, with $\gamma \in \Lambda^{k-1}\mv$ satisfying  \eqref{eq:Ag}; in particular, taking the sum over $x=e_i$ in this equation gives $A_*\gamma=0$. 

In order to show that $\alpha$ is a Killing form, we need to check that the two equations \eqref{eq:i1}--\eqref{eq:i2} in Proposition \ref{pro:killgen} hold. As already noticed, \eqref{eq:i1} is equivalent to \eqref{eq:Ag}. Moreover, \eqref{eq:i2} is trivially satisfied for $l=0,\ldots,k-3, k-1$ since $\alpha_{0},\ldots, \alpha_{k-2},\alpha_k$ vanish. Using the fact that $A_*\gamma=0$,  \eqref{eq:i2} for $l=k-2$ reads
$$
-\sum_{i=1}^n \lela [x,e_i],z_1\rira z_1 \wedge (e_i\lrcorner\ \gamma)=Ax\lrcorner\ (z_1\wedge \gamma).$$ 
It is easy to check that the left hand side equals $-z_1\wedge (Ax\lrcorner\ \gamma)$, and thus coincides with the  right hand side. Therefore, \eqref{eq:i2} for $l=k-2$ holds as well, and $\alpha$ is a Killing $k-$form.
\end{proof}

\begin{remark}\label{rem:g1}
Since $A$ is injective, \eqref{eq:Ag} has no solution for $\gamma\in\mv$, i.e. when $\alpha $ is a $2-$form. Therefore the Heisenberg Lie group endowed with any possible left-invariant metric, possesses no non-zero Killing $2-$forms; this fact was already known from \cite{BDS}.
\end{remark}

We can now give the precise description of Killing forms on $2-$step nilpotent metric Lie algebras with one-dimensional center. From now on we will use the same notation either for a skew-symmetric endomorphism $B\in \so(\mv)$ or the $2-$form $B=\frac12\sum_{i=1}^n e_i\wedge B e_i$, where $\{e_1, \ldots, e_n\}$ denotes an orthonormal basis  of $\mv$.

\begin{teo}\label{aa} The space $\mathcal{K}^k(\mn,g)$ of Killing $k-$forms on a $2-$step nilpotent metric Lie algebra $(\mn,g)$ with one-dimensional center is zero for $k$ even and is the real line spanned by $z_1\wedge A \wedge \ldots\wedge A$ for any $k$ odd with $k\le\dim(\mn)$. 
\end{teo}

\begin{proof} The form $\gamma:=A \wedge \ldots\wedge A$ clearly satisfies \eqref{eq:Ag}, so $z_1\wedge\gamma$ is Killing by Proposition \ref{pro:Ag}. 

Conversely, let $\alpha$ be a Killing $k-$form on $\mn$. By Proposition \ref{pro:Ag}, $\alpha=z_1\wedge \gamma$ where $\gamma\in \Lambda^{k-1}\mv$ verifies \eqref{eq:Ag}. 

We will prove by induction on $k$ that if a form $\gamma\in \Lambda^{k-1}\mv$ verifies \eqref{eq:Ag}, then $\gamma$ is a constant multiple of $A \wedge \ldots\wedge A$ (in particular $\gamma=0$ if $k$ is even).

For $k=1$ there is nothing to prove since $\gamma$ is just a constant, and for $k=2$ the claim follows from Remark \ref{rem:g1}. Assume now that $ k\geq 3$.

Polarizing \eqref{eq:Ag} we get
\begin{equation}\label{eq:Ag1}
 Ay\wedge (x\lrcorner \ \gamma)+Ax\wedge (y\lrcorner \ \gamma)=0,\quad \mbox{ for every  }x,y\in \mv.
\end{equation}
Considering an orthonormal basis $\{e_1, \ldots, e_n\}$ of $\mv$, and taking $y=A^{-1}e_i$ in \eqref{eq:Ag1} yields
$$e_i\wedge (x\lrcorner \ \gamma)+Ax\wedge(A^{-1}e_i\lrcorner \ \gamma)=0, \quad \mbox{ for every }x\in \mv, \;i=1, \ldots, n.$$ Taking the interior product with $e_i$ and summing over $i$ we get,
$$(n+3-k) (x\lrcorner \ \gamma) = Ax\wedge\sum_{i=1}^n  (e_i\lrcorner\ A^{-1}e_i\lrcorner\ \gamma),\quad \mbox{  for every }x\in \mv.$$
Define 
\begin{equation}\label{delta}\delta:=\frac1{n+3-k}\sum_{i=1}^n  e_i\lrcorner\ A^{-1}e_i\lrcorner\ \gamma,\end{equation} 
so that the equation above reads $x\lrcorner \ \gamma = Ax\wedge \delta$, for all $x\in \mv$. Using this equality, we obtain
\begin{equation}\label{gamma}\gamma=\frac1{k-1}\sum_{i=1}^ne_i \wedge (e_i\lrcorner\ \gamma)=\frac1{k-1}\sum_{i=1}^ne_i \wedge Ae_i\wedge \delta=\frac2{k-1}A\wedge\delta,\end{equation}
with $\delta \in \Lambda^{k-3}\mv$ defined by \eqref{delta}. Notice that for every $x\in \mv$  we have
\begin{equation*}
Ax\wedge (x\lrcorner\ \delta) = -x\lrcorner \ (Ax\wedge \delta)=-x\lrcorner\ (x\lrcorner\ \gamma) = 0.
\end{equation*}
This means that the $(k-3)-$form $\delta$ satisfies \eqref{eq:Ag}, so by the induction hypothesis, $\delta$ is a constant multiple of some exterior power of the $2-$form $A$ if $k$ is odd, and vanishes if $k$ is even. By \eqref{gamma}, the same holds for $\gamma$, thus finishing the proof.
\end{proof}

\begin{remark}
When the skew-symmetric endomorphism $A\in \so(\mv)$ defining the Lie algebra structure on $\mn=\mv\oplus \mz$ verifies $A^2=-{\rm Id_{\mv}}$, then the corresponding Riemannian metric on the Heisenberg Lie group admits a Sasakian structure whose associated Killing forms are the ones appearing in Theorem \ref{aa} (see also \cite{Se03}).

When $A^2\neq -{\rm Id_{\mv}}$, the corresponding Heisenberg manifold is not Sasakian but still carries non-trivial Killing forms of every odd degree.
\end{remark}

\section{Killing forms on $2-$step nilpotent Lie groups with two-dimensional center}

In this section we assume that the center of the 2--step nilpotent Lie algebra $\mn$ is two-dimensional.  Let $\{z_1,z_2\}$ be an orthonormal basis of $\mz$ and denote $A_i:=j(z_i)\in \so(\mv)$, for $i=1,2$. As in the previous section, we may identify the endomorphisms $A_i$ with their corresponding $2-$forms.

Let $\alpha$ be a $k-$form on $\mn$, then taking into account that $\dim \mz=2$ and using  \eqref{eq:lambdad}, we can write $\alpha=\alpha_{k}+\alpha_{k-1}+\alpha_{k-2}$. Because of Remark \ref{rem:evodd}, $\alpha$ defines a left-invariant Killing form if and only if both $\alpha^{\rm ev}=\alpha_{k}+\alpha_{k-2}$ and $\alpha^{\rm odd}=\alpha_{k-1}$ define Killing forms. 

\begin{pro}\label{k-1} Suppose that $(\mn,g)$ is irreducible and let $\alpha$ be a Killing $k-$form. If  $k\ge 4$, then $\alpha_{k-1}=0$.
\end{pro}
\begin{proof}
Since $\dim \mz=2$,  we can write $\alpha_{k-1}=z_1\wedge \gamma_1+z_2\wedge \gamma_2$ where $\gamma_1,\gamma_2\in\Lambda^{k-1}\mv$, which we assume not to be both zero.
If $\alpha$ is a Killing $k-$form then $\alpha_{k-1}$ is a Killing form and  Corollary \ref{cor:l1} implies that the above forms verify
\begin{eqnarray}
\label{eq:1} A_{i*}\gamma_j&=&0,\quad \mbox{ for all }i,j\in\{1,2\},\\
A_1x \lrcorner\ \gamma_2&=&A_2x \lrcorner\ \gamma_1, \quad   \mbox{ for all }x\in \mv.
\label{eq:2}
\end{eqnarray}
Moreover, polarizing \eqref{eq:i1} we obtain for every $x,y\in\mv$:
\begin{equation}
\label{eq:3}
A_1x\wedge (y \lrcorner\ \gamma_1) +A_1y\wedge (x \lrcorner\ \gamma_1)  + A_2 x\wedge (y \lrcorner\ \gamma_2)+A_2 y\wedge (x \lrcorner\ \gamma_2 )=0,\end{equation} 

We shall prove that these equations imply $[A_1,A_2]=0$, using the fact that $\gamma_1$ and $\gamma_2$ are not simultaneously zero. On the one hand, we apply $A_{1*}$ in \eqref{eq:2}, and using \eqref{eq:1} we obtain $(A_1^2x)\lrcorner\ \gamma_2=(A_1A_2x)\lrcorner\ \gamma_1$. In addition, replacing $x$ by $A_1x$ in \eqref{eq:2} gives $(A_1^2x)\lrcorner\ \gamma_2=(A_2A_1x)\lrcorner\ \gamma_1$ so these two combined  imply $[A_1,A_2]x\lrcorner\ \gamma_1=0$ for all $x\in \mv$. Similarly, 
\begin{equation}
\label{eq:a1}[A_1,A_2]x\lrcorner\ \gamma_2=0\qquad\mbox{ for all }x\in \mv. \end{equation}

On the other hand, the wedge product of  $A_1x$ and  equation \eqref{eq:3} reads
\begin{equation*}
A_1x\wedge A_1y\wedge  (x \lrcorner\ \gamma_1)  + A_1x\wedge A_2 x\wedge (y \lrcorner\ \gamma_2)+A_1x\wedge A_2 y\wedge (x \lrcorner\ \gamma_2) =0.
\end{equation*} 

Take $x=e_i$ and sum over $i$ this expression. Using \eqref{eq:1} and the fact that the skew-symmetric endomorphism corresponding to the $2$--form $\sum_i A_1e_i\wedge A_2 e_i$ is the commutator $[A_1,A_2]$, 
 we obtain that for every $y\in \mv$,
$[A_1,A_2]\wedge (y\lrcorner\ \gamma_2)=0$.

Taking in this equality the wedge product with $y$ for $y=e_i$ and summing over $i$ yields $[A_1,A_2]\wedge \gamma_2=0$. Consequently, for any $x\in \mv$ we have $$0=x\lrcorner\ \left([A_1,A_2]\wedge \gamma_2\right)=[A_1,A_2]x\wedge \gamma_2+[A_1,A_2]\wedge (x\lrcorner\ \gamma_2)=[A_1,A_2]x\wedge \gamma_2.$$ 

Contracting the last term with $[A_1,A_2]x$ and using \eqref{eq:a1} we get $|[A_1,A_2]x|^2\gamma_2=0$ for every $x\in \mv$. Similarly, $|[A_1,A_2]x|^2\gamma_1=0$. Since $\gamma_1$ and $\gamma_2$ are not simultaneously zero, we get $[A_1,A_2]=0$ as claimed.

Using the commutation of $A_1$ and $A_2$ we can suppose, after a rotation in $\mz$, that $\ker A_1\neq 0$. Indeed, there is an orthonormal basis of $\mv$ such that $A_1^2=diag(-a_1^2, \ldots, -a_n^2)$, $A_2^2=diag(-b_1^2, \ldots, -b_n^2)$ with $a_i,b_i\geq 0$. If $a_i=0$ for some $i$, then $A_1$ is already singular; if this is not the case but $b_i=0$ for some $i$ then we just interchange $A_1$ with $A_2$. Otherwise, let $t=\arctan (-a_1/b_1)$ and take $z_1'=\cos(t) z_1+\sin(t) z_2$, $z_2'=-\sin(t) z_1+\cos(t) z_2$. Then $\{z_1',z_2'\}$ is an orthonormal basis of $\mz$ and $A_i'=j(z_i')$ has a kernel of dimension at least two.

Assume from now on that $\ker A_1\neq 0$ and consider the orthogonal decomposition $\mv=\ker A_1\oplus {\rm Im }\ A_1$. Since $[A_1,A_2]=0$,  $A_2$ preserves this decomposition. Moreover, $A_2|_{{\rm Im}\ A_1}$ is non-zero since otherwise this would imply that $\R z_1\oplus{\rm Im }\ A_1$ and $\R z_2\oplus\ker A_1$ would be orthogonal ideals of $\mn$, contradicting the irreducibility hypothesis.

By \eqref{eq:2} we know that for any $x\in \ker A_1$, $A_2x\lrcorner\ \gamma_1=0$ which, together with the fact that $A_2|_{\ker A_1}$ is invertible, implies 
\begin{equation}
\label{eq:g1}
\gamma_1\in\Lambda^{k-1}({\rm Im }\ A_1).
\end{equation}

Let $y=A_1x\in {\rm Im}\ A_1$ with $x\in \mv$ arbitrary. From \eqref{eq:2} we have $y\lrcorner\  \gamma_2=A_2x\lrcorner\ \gamma_1$ and, by \eqref{eq:g1}, we know that $y\lrcorner\  \gamma_2\in \Lambda^{k-2}({\rm Im }\ A_1)$, therefore 
\begin{equation}
\label{eq:g2}
\gamma_2\in \Lambda^{k-1}({\rm Im }\ A_1)\oplus \Lambda^{k-1}(\ker A_1).
\end{equation} We shall prove that actually, $\gamma_2\in \Lambda^{k-1}(\ker A_1)$. Indeed, given $x\in \ker A_1$  and $y\in {\rm Im }\ A_1$, \eqref{eq:2} implies
$x\lrcorner\ \gamma_1=0$, so \eqref{eq:3} simplifies as
\begin{equation}\label{eq:4} 
A_2 x\wedge (y \lrcorner\ \gamma_2)+A_2 y\wedge (x \lrcorner\ \gamma_2) =0.
\end{equation}
By \eqref{eq:g2}, the first term in this equation belongs to  $\ker A_1\otimes \Lambda^{k-2}({\rm Im }\ A_1)$ and the second term belongs to $ {\rm Im }\ A_1\otimes \Lambda^{k-2}(\ker A_1)$. Since $k\geq 4$, both terms must vanish. 

In particular we proved that $A_2 x\wedge (y\lrcorner\ \gamma_2)=0$ for every $x\in \ker A_1$ and $y\in {\rm Im} A_1$. Recall that  $A_2x\neq 0$ for every $0\neq x\in \ker A_1$ (by Lemma \ref{lm:imjzv}). These two facts account to $y\lrcorner\ \gamma_2=0$ for every $y\in {\rm Im}\ A_1$ and thus, by \eqref{eq:g2}, $\gamma_2\in \Lambda^{k-1}(\ker A_1)$ as claimed.

Let $y\in {\rm Im}\ A_1$ be such that $A_2y\neq 0$, and let $x\in \ker A_1$. Then $A_2y\in {\rm Im}\ A_1$ and $x\lrcorner\ \gamma_2\in \Lambda^{k-2}(\ker A_1)$ and, since each term in \eqref{eq:4} vanishes, we have $A_2y\wedge (x\lrcorner\ \gamma_2)=0$. Therefore $x\lrcorner\ \gamma_2=0$ for every $x\in\ker A_1$, which implies $\gamma_2=0$.

We have shown that $\gamma_2=0$ so that $\alpha_{k-1}=z_1\wedge \gamma_1$. If moreover, $\ker A_2\neq 0$, then the same reasoning as above gives $\gamma_1=0$. In the case where $A_2$ is non-singular, we proceed as before and by using a rotation, we make $A_2$ a singular endomorphism. Indeed, we take $z_1'=-\sin(t) z_2+\cos(t) z_1$, $z_2'=\cos(t)z_2+\sin(t)z_1$ with $t=\arctan(-b_i/a_i)$ for some $i$ such that $a_i\neq 0$. In this new basis, $j(z_2')$ has a non-trivial kernel and $\alpha_{k-1}=\cos(t)z_1'\wedge \gamma_1+\sin(t)z_2'\wedge \gamma_1$ with $\cos(t)\neq 0$. Therefore, the argument above implies that $\gamma_1=0$, which is a contradiction.
\end{proof}

We will now prove the main result of this section. 
\begin{teo}\label{k} If $(\mn,g)$ is an irreducible $2-$step nilpotent metric Lie algebra with two-dimensional center and $4\le k\le \dim(\mn)-1$, then every left-invariant Killing $k-$form on the associated simply connected Riemannian Lie group $(N,g)$ vanishes.
\end{teo}

\begin{proof}
Let $\alpha=\alpha_k+\alpha_{k-1}+\alpha_{k-2}$ be the decomposition of a Killing form according to \eqref{eq:lambdad}. By Proposition \ref{k-1}, $\alpha_{k-1}=0$. We will show that $\alpha_k+\alpha_{k-2}$, vanishes too for $k\ge 3$. 

In order to simplify the notation, we will denote $\alpha_k=:\beta\in\Lambda^k\mv$ and $\alpha_{k-2}=:z_1\wedge z_2\wedge\delta$, with $\delta\in\Lambda^{k-2}\mv$. By Lemma \ref{lm:imjzv} we have 
\begin{equation}\label{e0} \Im A_1+ \Im A_2=\mv,
\end{equation}
Moreover, since $\alpha_{k-4}=0$, Corollary \ref{cor:l1} gives
\begin{equation}\label{e1} A_{1*}\delta=A_{2*}\delta=0,
\end{equation}
Finally, \eqref{eq:i2} for $l=k-1$ is equivalent to 
\begin{equation}\label{e2} A_1x\lrcorner\ \beta=x\lrcorner\  A_{1*}\beta-A_2x\wedge\delta,\qquad\forall\ x\in\mv,
\end{equation}
\begin{equation}\label{e3} A_2x\lrcorner\ \beta=x\lrcorner\  A_{2*}\beta+A_1x\wedge\delta,\qquad\forall\ x\in\mv.
\end{equation}

The fact that $\alpha_k+\alpha_{k-2}=0$ for $k\ge 3$ is now a direct consequence of the following general result of linear algebra.
\end{proof}

\begin{pro}\label{p63} Let $2\le k$ be an integer, $(\mv,\langle\cdot,\cdot\rangle)$ a Euclidean vector space, $A_1,A_2\in \so(\mv)$, $\beta\in\Lambda^k\mv$, and $\delta\in\Lambda^{k-2}\mv$ such that \eqref{e0}--\eqref{e3} are verified. 
If $\beta$ and $\delta$ are not both zero, then $k=2$ or $k=2+\dim(\mv)$.
\end{pro}
\begin{proof} If $k\ge 2+\dim(\mv)$ there is nothing to prove, so we can assume for the rest of the proof that $k\le1+\dim(\mv)$.

Let $\{e_1, \ldots, e_n\}$ be an orthonormal basis of $\mv$. Taking the wedge product with $x$ in \eqref{e2} and summing over the orthonormal basis with $x=e_i$ yields $-A_{1*}\beta=kA_{1*}\beta-2A_2\wedge\delta$, whence 
\begin{equation}\label{e4} A_{1*}\beta=\frac2{k+1}A_2\wedge\delta.
\end{equation} Similarly, using \eqref{e3} we get 
\begin{equation}\label{e5} A_{2*}\beta=-\frac2{k+1}A_1\wedge\delta.
\end{equation}

Notice that the two formulas above also follow from \eqref{eq:i3}.  Using them, \eqref{e2}--\eqref{e3} become
\begin{equation}\label{e6} (A_1x)\lrcorner\ \beta=\frac2{k+1}x\lrcorner\ (A_2\wedge\delta)-(A_2x)\wedge\delta,\qquad\forall\ x\in\mv,
\end{equation}

\begin{equation}\label{e7} (A_2x)\lrcorner\ \beta=-\frac2{k+1}x\lrcorner\ (A_1\wedge\delta)+(A_1x)\wedge\delta,\qquad\forall\ x\in\mv.
\end{equation}

We claim that 
\begin{equation}\label{e8}[A_1,A_2]\wedge\delta\ne 0.
\end{equation}
Indeed, assuming for a contradiction that $[A_1,A_2]\wedge\delta=0$, and applying $A_{1*}$ to \eqref{e4} yields $A_{1*}A_{1*}\beta=0$, so taking the scalar product with $\beta$ and using the fact that $A_{1*}$ is skew-symmetric on $\Lambda^k\mv$ yields $A_{1*}\beta=0$. Similarly, applying $A_{2*}$ to \eqref{e5} gives $A_{2*}\beta=0$. Using \eqref{e2} we compute in $e_i$:
\begin{eqnarray*}
\sum_{i=1}^n|(A_1e_i)\lrcorner\ \beta|^2&=&-\sum_{i=1}^n\langle (A_1e_i)\lrcorner\ \beta,(A_2e_i)\wedge\delta\rangle=-\sum_{i=1}^n\langle \beta,(A_1e_i)\wedge (A_2e_i)\wedge\delta\rangle\\
&=&\langle \beta,[A_1,A_2]\wedge\delta\rangle=0,
\end{eqnarray*}
thus showing that $(A_1x)\lrcorner\  \beta=0$ for every $x\in\mv$. Similarly, using \eqref{e3} we would obtain $(A_2x)\lrcorner\  \beta=0$ for every $x\in\mv$. From \eqref{e0} this would imply $\beta=0$; by using again \eqref{e2}--\eqref{e3} and \eqref{e0} we also get $x\wedge \delta=0$ for every $x\in\mv$. Since the degree of $\delta$ is $k-2\le\dim(\mv)-1$, this would imply $\delta=0$, thus contradicting our hypothesis that $\beta$ and $\delta$ are not both zero. This proves \eqref{e8}.

\begin{lm}\label{lm:delta} If $\xi\in\mv$ satisfies $\xi\wedge\delta=0$ then $\xi=0$.
\end{lm}
\begin{proof}
If $\xi\wedge\delta=0$, then for every $x\in \mv$ and every exterior form $\omega$ we have 
\begin{equation}\label{e9}\xi\wedge(x\lrcorner\ (\omega\wedge\delta))=\langle x,\xi\rangle\omega\wedge\delta.
\end{equation}
Applying $A_{1*}$ to \eqref{e6} and subtracting the same equation with $x$ replaced by $A_1x$, yields 
\begin{equation}\label{e10} A_1x\lrcorner\  A_{1*}\beta=\frac2{k+1}x\lrcorner\ ([A_1,A_2]\wedge\delta)-[A_1,A_2]x\wedge\delta,\qquad\forall\ x\in\mv.
\end{equation}
Taking the wedge product with $\xi$ in this equation and using \eqref{e4} and \eqref{e9}, we finally get
$$\langle A_1x,\xi\rangle A_2\wedge\delta=\langle x,\xi\rangle [A_1,A_2]\wedge\delta,\qquad\mbox{ for all} \ x\in\mv.$$
For $x=\xi$, the left-hand term vanishes, so $|\xi|^2[A_1,A_2]\wedge\delta=0$. Together with \eqref{e8}, this concludes the proof of the lemma.
\end{proof}

Consider now the Lie algebra $\mg\subset\so(\mv)$ generated by $A_1$ and $A_2$, endowed with the ad-invariant scalar product inherited from $\so(\mv)$. By \eqref{e1}, $A_*\delta=0$ for every $A\in \mg$ so the subspace 
$$I:=\{A\in \mg\ |\ A\wedge\delta=0\}$$
is clearly and ideal of $\mg$.

\begin{lm}\label{i} The ideal $I$ vanishes (i.e. $\wedge\delta:\Lambda^2\mv\to\Lambda^k\mv$ is injective).
\end{lm}

\begin{proof} Since the scalar product on $\mg$ is ad-invariant, the orthogonal $I^\perp$ of $I$ is also an ideal of $\mg$, and thus the elements of $I$ commute with the elements of $I^\perp$. For every $A\in\mg$ we write $A=A^I+A^\perp$, where $A^I$ and $A^\perp$ are the orthogonal projections of $A$ on $I$ and $I^\perp$. Note that, by the definition of $\mg$, $I$ is generated by $A_1^I$ and $A_2^I$.

Acting with $A_{1*}$ and $A_{2*}$ on \eqref{e4}--\eqref{e5} and using that $A_*\delta=0$ for all $A\in\mg$, we get by a straightforward inductive argument that for every $A\in \mg$ there exists $B\in \mg$ such that $A_*\beta=B\wedge\delta$. We decompose $\beta$ with respect to the orthogonal direct sum 
\begin{equation}\label{dec}\Lambda^k\mv=(\mg\wedge\delta)\oplus (\mg\wedge\delta)^\perp
\end{equation}
as $\beta=:A_3\wedge\delta+\beta_0$. By the definition of $I$, we can choose $A_3\in I^\perp$. This choice will ensure that $[A_3,A]=[A_3,A]^\perp$ for every $A\in \mg$.

The action of every element of $\mg$ clearly preserves the decomposition \eqref{dec}, we thus get $A_*\beta_0=0$ for every $A\in\mg$.

Equations \eqref{e4}--\eqref{e5} then become
$$ A_{1*}(A_3\wedge\delta)=\frac2{k+1}A_2\wedge\delta,\qquad A_{2*}(A_3\wedge\delta)=-\frac2{k+1}A_1\wedge\delta,$$
which also read
\begin{equation}\label{a3}[A_1,A_3]=[A_1,A_3]^\perp=\frac2{k+1}A_2^\perp,\qquad [A_2,A_3]=[A_2,A_3]^\perp=-\frac2{k+1}A_1^\perp.
\end{equation}

Now, using the fact that $A_{3*}\beta=0$, acting on \eqref{e6}--\eqref{e7} with $A_{3*}$, subtracting the same equations with $x$ replaced by $A_3x$ and using \eqref{a3}, yields 
\begin{equation}\label{e16} -A^\perp_2x\lrcorner\ \beta=\frac2{k+1}x\lrcorner\ (A^\perp_1\wedge\delta)-A^\perp_1x\wedge\delta,\qquad\mbox{ for all}\ x\in\mv,
\end{equation}

\begin{equation}\label{e17} A^\perp_1x\lrcorner\ \beta=\frac2{k+1}x\lrcorner\ (A^\perp_2\wedge\delta)-A^\perp_2x\wedge\delta,\qquad\mbox{ for all}\ x\in\mv.
\end{equation}

Making the sum of \eqref{e7} and \eqref{e16} and the difference between \eqref{e6} and \eqref{e17}, and using the fact that $A^I_1\wedge\delta=A^I_2\wedge\delta=0$, we obtain
\begin{equation}\label{e26} A^I_1x\lrcorner\ \beta=-A^I_2x\wedge\delta,\qquad\mbox{ for all}\ x\in\mv,
\end{equation}

\begin{equation}\label{e27} A^I_2x\lrcorner\ \beta=A^I_1x\wedge\delta,\qquad\mbox{ for all}\ x\in\mv.
\end{equation}

Since  $A_3$ is in $I^\perp$, it commutes with $A^I_1$, so $A_{1*}^I\beta=A_{1*}^I(A_3\wedge\delta)=0$. Applying $A_{1*}^I$ to \eqref{e26} and subtracting the same equation with $x$ replaced by $A_1x$, we thus get
\begin{equation}\label{e36} [A_1^I,A^I_2]x\wedge\delta=\frac2{k+1}x\lrcorner\ ([A_1^I,A^I_2]\wedge\delta),\qquad\mbox{ for all}\ x\in\mv.
\end{equation}
On the other hand, $[A_1^I,A^I_2]\wedge\delta=0$ by the definition of $I$, so Lemma \ref{lm:delta} gives
\begin{equation}\label{ai}[A_1^I,A^I_2]=0.
\end{equation}

We now replace $x$ by $-A_2^Ix$ in \eqref{e26}, and by $A_1^Ix$ in \eqref{e27}. The sum of the resulting equation reads
$$(-A^I_1A^I_2x+A^I_2A^I_1x)\lrcorner\ \beta=((A^I_2)^2x+(A^I_1)^2x)\wedge\delta,\qquad\mbox{ for all}\ x\in\mv.$$
The left-hand term vanishes by \eqref{ai}, so by Lemma \ref{lm:delta} again we get $(A^I_2)^2+(A^I_1)^2=0$. Since $A_1^I$ and $A_2^I$ are skew-symmetric, this shows that they both vanish, so $I$, which is generated by $A_1^I$ and $A_2^I$, has to vanish too.
\end{proof}

Since $I=0$, \eqref{a3} becomes 
\begin{equation}\label{a4}[A_1,A_3]=\frac2{k+1}A_2,\qquad [A_2,A_3]=-\frac2{k+1}A_1.
\end{equation}

Applying $A_{2*}$ to \eqref{e6}, subtracting \eqref{e6} with $x$ replaced by $A_2x$, and using \eqref{e5}, yields 
\begin{equation}\label{e41}[A_2,A_1]x\lrcorner\ \beta- \frac2{k+1}A_1x\lrcorner\  (A_1\wedge\delta)=0,\qquad\mbox{ for all}\ x\in\mv.
\end{equation}
Similarly, applying $A_{1*}$ to \eqref{e7}, subtracting \eqref{e7} with $x$ replaced by $A_1x$, and using \eqref{e4}, we obtain
\begin{equation}\label{e42}[A_1,A_2]x\lrcorner\ \beta+ \frac2{k+1}A_2x\lrcorner\  (A_2\wedge\delta)=0,\qquad\mbox{ for all}\ x\in\mv,
\end{equation}
whence 
\begin{equation}\label{e43}A_1x\lrcorner\  (A_1\wedge\delta)=A_2x\lrcorner\  (A_2\wedge\delta)=0,\qquad\mbox{ for all}\ x\in\mv.
\end{equation}

We replace $x$ by $A_2x$ in \eqref{e6}, and by $A_1x$ in \eqref{e7}. The sum and the difference of the resulting equations read
\begin{eqnarray*}(A_1A_2x+A_2A_1x)\lrcorner\ \beta&=&\frac2{k+1}A_2x\lrcorner\  (A_2\wedge\delta)- \frac2{k+1}A_1x\lrcorner\  (A_1\wedge\delta)\\&&+((A_1)^2x-(A_2)^2x)\wedge\delta,\qquad\mbox{ for all}\ x\in\mv,
\end{eqnarray*}
\begin{eqnarray*}(A_1A_2x-A_2A_1x)\lrcorner\ \beta&=&\frac2{k+1}A_2x\lrcorner\  (A_2\wedge\delta)+\frac2{k+1}A_1x\lrcorner\  (A_1\wedge\delta)\\&&-((A_2)^2x+(A_1)^2x)\wedge\delta,\qquad\mbox{ for all}\ x\in\mv,
\end{eqnarray*}
Taking \eqref{e41}--\eqref{e42} into account, these equations can be simplified as 
\begin{equation}\label{e51}((A_1)^2x-(A_2)^2x)\wedge\delta=(A_1A_2x+A_2A_1x)\lrcorner\ \beta,\qquad\mbox{ for all}\ x\in\mv,\end{equation}
\begin{equation}\label{e52}((A_2)^2x+(A_1)^2x)\wedge\delta=\frac6{k+1}A_1x\lrcorner\  (A_1\wedge\delta)=\frac6{k+1}A_2x\lrcorner\  (A_2\wedge\delta),\qquad\mbox{ for all}\ x\in\mv.\end{equation}
Since the second term of \eqref{e52} is invariant by $A_{1*}$ and the third term is invariant by $A_{2*}$, Lemma \ref{lm:delta} shows that $(A_1)^2+(A_2)^2$ commutes with $A_1$ and $A_2$. Consequently, both $(A_1)^2$ and $(A_2)^2$ commute with $A_1$ and $A_2$, so they commute with every element of $\mg$.

In particular $A_3$ commutes with $(A_1)^2$, which by \eqref{a4} yields $A_1A_2+A_2A_1=0$. From \eqref{e51} together with Lemma \ref{lm:delta}, we obtain that $(A_1)^2=(A_2)^2$. Moreover, \eqref{e0} implies that $A_1$ and $A_2$ are both invertible. 

We now replace $x$ by $(A_1)^{-1}x$ in \eqref{e52} and obtain
$$A_1x\wedge\delta=\frac3{k+1}x\lrcorner\  (A_1\wedge\delta),\qquad\mbox{ for all}\ x\in\mv.$$

Taking the wedge product with $x$ in this last equation and summing over $i$ with $x=e_i$ yields
$$2A_1\wedge\delta=\frac{3k}{k+1}A_1\wedge\delta,$$
whence $(k-2)A_1\wedge\delta=0$. Since $A_1\ne 0$ (e.g. by \eqref{e8}), Lemma \ref{i} implies that $k=2$, thus finishing the proof of the proposition.
\end{proof}

Using the descriptions of Killing $2-$ and $3-$forms on $2-$step nilpotent Lie groups obtained in \cite{dBM2}, we can summarize the above arguments in the following classification result:
\begin{teo}
Let $(\mn,g)$ be a $2-$step nilpotent metric Lie algebra with two-dimensional center $\mz$. If $(\mn,g)$ is irreducible, then the space of Killing forms on $(\mn,g)$ satisfies:
\begin{itemize}
\item $\mathcal{K}^1(\mn,g)=\mz$ is two-dimensional.
\item $\mathcal{K}^2(\mn,g)$ is one-dimensional if $\mn$ has a bi-invariant $g$-orthogonal complex structure, and zero otherwise.
\item $\mathcal{K}^3(\mn,g)$ is one-dimensional if $(\mn,g)$ is naturally reductive, and zero otherwise.
\item $\mathcal{K}^k(\mn,g)=0$ for $4\le k\le \dim(\mn)-1$.
\item $\mathcal{K}^k(\mn,g)=\Lambda^k\mn$ is one-dimensional if $k=\dim(\mn)$.
\end{itemize}
If $(\mn,g)=(\mn_1,g_1)\oplus(\mn_2,g_2)$ is reducible, with $\dim(\mn_1)\le\dim(\mn_2)$, then the space of Killing forms on $(\mn,g)$ satisfies:
\begin{itemize}
\item $\mathcal{K}^1(\mn,g)=\mz$ is two-dimensional.
\item $\mathcal{K}^k(\mn,g)=0$ if $k$ is even and $2\le k<\dim(\mn)$.
\item If $k$ is odd and $3\le k\le \dim(\mn)-1$, then $\mathcal{K}^k(\mn,g)$ is two-dimensional for $k\le\dim(\mn_1)$, one-dimensional for $\dim(\mn_1)<k\le\dim(\mn_2)$ and zero for $k>\dim(\mn_2)$.
\item $\mathcal{K}^k(\mn,g)=\Lambda^k\mn$ is one-dimensional if $k=\dim(\mn)$.
\end{itemize}
\end{teo}

\bibliographystyle{plain}
\bibliography{biblio}

\begin{thebibliography}{10}

\bibitem{AD}
A.~{Andrada} and I.~{Dotti}.
\newblock {Conformal Killing-Yano 2-forms.}
\newblock {\em {Differ. Geom. Appl.}}, 58:103--119, 2018.

\bibitem{BDS}
M.L. {Barberis}, I.~{Dotti}, and O.~{Santill\'an}.
\newblock {The Killing-Yano equation on Lie groups.}
\newblock {\em {Classical Quantum Gravity}}, 29(6):1--10, 2012.

\bibitem{lau}
M.L. Barberis, A.~Moroianu, and U.~Semmelmann.
\newblock {Generalized vector cross products and Killing forms on negatively
  curved manifolds.}
\newblock {\em {Geom. Dedicata}}, 205(1):113--127, 2020.

\bibitem{bms}
F.~Belgun, A.~Moroianu, and U.~Semmelmann.
\newblock {Killing Forms on Symmetric Spaces.}
\newblock {\em {Diff. Geom. Appl.}}, 24:215--222, 2006.

\bibitem{BE-GO}
C.~Benson and C.~Gordon.
\newblock K\"ahler and symplectic structures on nilmanifolds.
\newblock {\em Topology}, 27(4):513--518, 1988.

\bibitem{Be86}
A.~Besse.
\newblock {\em Einstein manifolds}.
\newblock Springer-Verlag, Berlin, 1986.

\bibitem{dBM2}
V.~del Barco and A.~Moroianu.
\newblock Killing forms on $2$-step nilmanifolds.
\newblock {\em J. Geom. Anal.}, 2019.
\newblock https://doi.org/10.1007/s12220-019-00304-1.

\bibitem{dBM}
V.~{del Barco} and A.~{Moroianu}.
\newblock {Symmetric Killing tensors on nilmanifolds}.
\newblock {\em Bull. Soc. Math. France}, 178(3):411--438, 2020.

\bibitem{EB}
P.~Eberlein.
\newblock Geometry of {$2$}-step nilpotent groups with a left invariant metric.
\newblock {\em Ann. Sci. \'Ecole Norm. Sup. (4)}, 27(5):611--660, 1994.

\bibitem{GO2}
C.~Gordon.
\newblock {Naturally reductive homogeneous Riemannian manifolds.}
\newblock {\em {Canad. J. Math.}}, 37:467--487, 1985.

\bibitem{KO-NO}
S.~Kobayashi and K.~Nomizu.
\newblock {\em Foundations of Differential Geometry}.
\newblock John Wiley and Sons, 1963.

\bibitem{Mil76}
J.~{Milnor}.
\newblock {Curvatures of left invariant metrics on Lie groups.}
\newblock {\em {Adv. Math.}}, 21:293--329, 1976.

\bibitem{au}
A.~Moroianu and U.~Semmelmann.
\newblock {Killing Forms on Quaternion-K\"ahler Manifolds.}
\newblock {\em {Ann. Global Anal. Geom.}}, 28:319--335, 2005.

\bibitem{ms08}
A.~Moroianu and U.~Semmelmann.
\newblock {Twistor forms on Riemannian Products}.
\newblock {\em {J. Geom. Phys.}}, 58:1343--1345, 2008.

\bibitem{Se03}
U.~Semmelmann.
\newblock {Conformal Killing forms on Riemannian manifolds}.
\newblock {\em {Math. Z.}}, 245:503--527, 2003.

\bibitem{s}
U.~Semmelmann.
\newblock {Killing forms on ${\rm G}\sb 2$- and ${\rm Spin}\sb 7$-manifolds.}
\newblock {\em {J. Geom. Phys.}}, 56:1752--1766, 2006.

\bibitem{y75}
S.~Yamaguchi.
\newblock {On a Killing $p$-form in a compact Kählerian manifold.}
\newblock {\em {Tensor (N.S.).}}, 29:274--276, 1975.

\end{thebibliography}
\end{document}